\documentclass[aap,MSNbibl,nameyear,seceqn,dvips]{arximspdf}
\usepackage{dcolumn}
\usepackage{graphicx}
\doi{10.1214/12-AAP884} 
\volume{23}
\issue{4}
\pubyear{2013}
\firstpage{1692}
\lastpage{1720}

\makeatletter
\newcolumntype{d}[1]{D{.}{.}{#1}}
\newcommand{\rrvert}{\vert}
\newcommand{\llvert}{\vert}
\newcommand{\eqref}[1]{(\ref{#1})}
\renewcommand{\citep}[1]{(\citeauthor{#1} \citeyear{#1})}
\newtheorem{thmm}{Theorem}[section]
\newtheorem{prop}[thmm]{Proposition}
\newtheorem{cor}[thmm]{Corollary}
\newproclaim{rem}{Remark}
\newproclaim{rems}{Remarks}
\newproclaim{defn}{Example}
\newproclaim{desc}{Description}

\newcommand{\tv}{\operatorname{TV}}
\newcommand{\separ}{\operatorname{sep}}
\newcommand{\var}{\operatorname{var}}

\newcommand{\pp}{\mathbb{P}}
\makeatother

\begin{document}
\begin{frontmatter}

\title{Analysis of casino shelf shuffling machines}
\runtitle{Shelf shuffling}

\begin{aug}
\author[A]{\fnms{Persi} \snm{Diaconis}\thanksref{t1}\ead[label=e1]{diaconis@math.stanford.edu}},
\author[B]{\fnms{Jason} \snm{Fulman}\thanksref{t2}\ead[label=e2]{fulman@usc.edu}}
\and
\author[C]{\fnms{Susan} \snm{Holmes}\corref{}\thanksref{t3}\ead[label=e3]{susan@stat.stanford.edu}}
\address[A]{P. Diaconis\\
Departments of Mathematics and Statistics\\
Stanford University\\
Stanford, California 94305-4065\\
USA\\
\printead{e1}}
\address[B]{J. Fulman\\
Department of Mathematics\\
University of Southern California\\
Los Angeles, California 90089-2532\\
USA\\
\printead{e2}}
\address[C]{S. Holmes\\
Department of Statistics\\
Stanford University\\
Stanford, California 94305-4065\\
\printead{e3}}

\thankstext{t1}{Supported in part by NSF Grant DMS-08-04324.}
\thankstext{t2}{Supported in part by NSF Grant DMS-08-02082
and National Security Agency Grant H98230-08-1-0133.}
\thankstext{t3}{Supported in part by National Institutes of Health
Grant R01 GM086884.}
\runauthor{P. Diaconis, J. Fulman and S. Holmes}
\affiliation{Stanford University, University of Southern California and
Stanford~University}
\end{aug}

\received{\smonth{8} \syear{2011}}
\revised{\smonth{6} \syear{2012}}

%
\begin{abstract}
Many casinos routinely use mechanical card shuffling machines. We
were asked to evaluate a new product, a shelf shuffler. This leads
to new probability, new combinatorics and to some practical advice
which was adopted by the manufacturer. The interplay between theory,
computing, and real-world application is developed.
\end{abstract}

%

\end{frontmatter}

\section{Introduction}\label{sec1}

We were contacted by a manufacturer of casino equipment to evaluate a
new design for a casino card-shuffling machine. The machine, already
built, was a sophisticated ``shelf shuffler'' consisting of an opaque
box containing ten shelves. A deck of cards is dropped into the top of
the box. An internal elevator moves the deck up and down within the
box. Cards are sequentially dealt from the bottom of the deck onto the
shelves; shelves are chosen uniformly at random at the command of a
random number generator. Each card is randomly placed above or below
previous cards on the shelf with probability $1/2$. At the end, each
shelf contains about $1/10$ of the deck. The ten piles are now assembled
into one pile, in random order. The manufacturer wanted to know if one
pass through the machine would yield a well-shuffled deck.

Testing for randomness is a basic task of statistics. A standard
approach is to design some ad hoc tests such as: Where do the original
top and bottom cards wind up? What is the distribution of cards that
started out together? What is the distribution, after one shuffle, of
the relative order of groups of consecutive cards? Such tests
\textit{had} been carried out by the engineers who designed the
machine, and seemed satisfactory.

We find closed-form expressions for the probability of being at a
given permutation after the shuffle. This gives exact expressions for
various global distances to uniformity, for example, total variation. These
suggest that the machine has flaws. The engineers (and their bosses)
needed further convincing; using our theory, we were able to show that
a knowledgeable player could guess about 9 $1/2$ cards correctly in a
single run through a 52-card deck. For a well-shuffled deck, the
optimal strategy gets about 4 $1/2$ cards correct. This data
\textit{did} convince the company. The theory also suggested a useful
remedy. Journalist accounts of our shuffling adventures can be found
in \citeauthor{klarreich} (\citeyear{klarreich2,klarreich}), \citet{mackenzie}.

Section~\ref{sec2} gives background on casino shufflers, needed probability
and the literature of shuffling. Section~\ref{sec3} gives an analysis
of a
single shuffle; we give a closed formula for the chance that a deck of
$n$ cards passed through a machine with $m$ shelves is in final order
$w$. This is used to compute several classical distances to
randomness. In particular it is shown that, for $n$ cards, the
$l_\infty$ distance is asymptotic to $e^{1/12c^2}-1$ if the number of
shelves $m=cn^{3/2}$ and $n$ is large. The combinatorics of shelf
shufflers turns out to have connections to the ``peak algebra'' of
algebraic combinatorics. This allows nice formulas for the
distribution of several classical test statistics: the cycle structure
(e.g., the number of fixed points), the descent structure and the
length of the longest increasing subsequence.

Section~\ref{sec4} develops tools for analyzing repeated shelf shuffling.
Section~\ref{sec5} develops our ``how many can be correctly guessed'' tests.
This section also contains our final conclusions.

\section{Background}\label{sec2}

This section gives background and a literature review. Section~\ref{sec21}
treats shuffling machines; Section~\ref{sec22} gives probability background;
Section~\ref{sec23} gives an overview of related literature and results
on the
mathematics of shuffling cards.

\subsection{Card shuffling machines}\label{sec21}

Casinos worldwide routinely employ mechanical card-shuffling machines
for games such as blackjack and poker. For example, for a single deck
game, two decks are used. While the dealer is using the first deck in
the usual way, the shuffling machine mixes the second deck. When the
first deck is used up (or perhaps half-used), the second deck is
brought into play and the first deck is inserted into the machine.
Two-, four-, and six-deck machines of various designs are also in
active use.

The primary rationale seems to be that dealer shuffling takes time and
use of a machine results in approximately 20\% more hands per hour.
The machines may also limit dealer cheating.

The machines in use are sophisticated, precision devices, rented to
the casino (with service contracts) for approximately \$500 per month
per machine. One company told us they had about 8000 such machines in
active use; this amounts to millions of dollars per year. The
companies involved are substantial businesses, listed on the New York
Stock Exchange.

One widely used machine simulates an ordinary riffle shuffle by
pushing two halves of a single deck together using mechanical pressure
to make the halves interlace. The randomness comes from slight
physical differences in alignment and pressure. In contrast, the shelf
shufflers we analyze here use computer-generated pseudo-random numbers
as a source of their randomness.

The pressure shufflers require multiple passes (perhaps seven to ten)
to adequately mix 52 cards. Our manufacturer was keen to have a single
pass through suffice.

\subsection{Probability background}\label{sec22}

Let $S_n$ denote the group of permutations of $n$ objects. Let
$U(\sigma)=1/n!$ denote the uniform distribution on $S_n$. If $P$ is a
probability on $S_n$, the total variation, separation, and $l_\infty$
distances to uniformity are
\begin{eqnarray}
\label{pound} \|P-U\|_{\tv}&=&\frac12\sum_w\bigl|P(w)-U(w)\bigr|=
\max_{A\subseteq S_n}\bigl|P(A)-U(A)\bigr|\nonumber\\
&=&\frac12\max_{\|f\|_\infty\leq
1}\bigl|P(f)-U(f)\bigr|,
\\
\separ(P)&=&\max_w \biggl(1-\frac{P(w)}{U(w)} \biggr),\qquad\|P-U
\|_\infty=\max_w\biggl\llvert1-\frac{P(w)}{U(w)}\biggr
\rrvert.\nonumber
\end{eqnarray}
Note that $\|P-U\|_{\tv}\leq\separ(P)\leq\|P-U\|_\infty$. The first two
distances are less than~1; the $\|\cdot\|_\infty$ norm can be as large as
$n!-1$.

If one of these distances is suitably small, then many test statistics
evaluate to approximately the same thing under $P$ and $U$. This gives
an alternative to ad hoc tests. The methods developed below allow
exact evaluation of these and many further distances (e.g., chi-square
or entropy).

Repeated shuffling is modeled by convolution,
\[
P\ast P(w)=\sum_vP(v)P\bigl(wv^{-1}
\bigr),\qquad P^{*k}(w)=P\ast P^{*(k-1)}(w).
\]

All of the shelf shufflers generate ergodic Markov chains (even if
only one shelf is involved), and so $P^{*k}(w)\to U(w)$ as
$k\to\infty$. One question of interest is the quantitative measurement
of this convergence using one of the metrics above.

\subsection{Previous work on shuffling}\label{sec23}

\subsubsection*{Early work}

The careful analysis of repeated shuffles of a deck of cards has
challenged probabilists for over a century.
The first efforts were made by Hadamard\vadjust{\goodbreak} (\citeyear{Ha06}) in his review of Gibbs book on statistical mechanics.
Later,
Poincar\'{e}~(\citeyear{Po12}) studied the problem. These great mathematicians
proved that \textit{in principle} repeated shuffling \textit{would}
mix cards at an exponential rate but gave no examples or quantitative
methods to get useful numbers in practical problems.

\citet{borel} studied riffle shuffling and concluded heuristically
that about seven shuffles would be required to mix 52 cards. Emile
Borel also reported joint work with Paul Levy, one of the great
probabilists of the twentieth century; they posed some problems but
were unable to make real progress.

Isolated but serious work on shuffling was reported in a 1955 Bell
Laboratories report by Edgar Gilbert. He used information theory to
attack the problems and gave some tools for riffle shuffling developed
jointly with Claude Shannon.

They proposed what has come to be called the Gilbert--Shannon--Reeds
model for riffle shuffling; this presaged much later work.
\citet{thorp} proposed a less realistic model and showed how poor
shuffling could be exploited in casino games. Thorp's model is
analyzed in \citet{morris}. \citet{epstein} reports practical studies
of how casino dealers shuffle with data gathered with a very precise
microphone! The upshot of this work was a well-posed mathematics
problem and some heuristics; further early history appears in Chapter
4 of \citet{D88}.

\subsubsection*{The modern era}

The modern era in quantitative analysis of shuffling begins with
papers of \citet{DS} and \citet{aldous}. They introduced rigorous
methods, Fourier analysis on groups and coupling. These gave sharp
upper and lower bounds, suitably close, for real problems. In
particular, Aldous sketched out a proof that $\frac32\log_2n$ riffle
shuffles mixed $n$ cards. A more careful argument for riffle shuffling
was presented by \citet{AD}. This introduced ``strong stationary
times,'' a powerful method of proof which has seen wide
application. It is applied here in Section~\ref{sec4}.

A definitive analysis of riffle shuffling was finally carried out in
\citet{bayer} and \citet{DMP}. They were able to derive simple
closed-form expressions for all quantities involved and do exact
computations for $n=52$ (or 32 or 104 or $\ldots$). This results
in the ``seven shuffles theorem'' explained below. A clear elementary
account of these ideas is in Mann (\citeyear{mann94,mann95}) reprinted in
\citet{grin}. See \citet{ethier} for an informative textbook account.

The successful analysis of shuffling led to a host of developments,
the techniques refined and extended. For example, it is natural to
want not only the order of the cards, but also the ``up-down pattern''
of one-way backs to be randomized. Highlights include work of
\citet{bidi} and \citet{brown} who gave a geometric interpretation of
shuffling which had many extensions to which the same analysis
applied. \citeauthor{lalley1} (\citeyear{lalley1,lalley2}) studied less random methods of riffle
shuffling. \citeauthor{F4} (\citeyear{F4,F3,F1}) showed that interspersing cuts does not
materially effect things and gave high level explanations for
miraculous accidents connecting shuffling and Lie theory. The work is
active and ongoing. Recent surveys are given by
\citeauthor{D96} (\citeyear{D96,D98}), \citet{fulman98,oconn}.

In recent work, \citet{DMP}, Conger and Vis\-wanath (\citeyear{conger2}) and \citet{assaf} have studied the
number of shuffles required to have selected features randomized
(e.g., the original top card, or the values but not the suits). Here,
fewer shuffles suffice. \citet{conger} shows that the way the cards
are dealt out after shuffling affects things. The mathematics of
shuffling is closely connected to modern algebraic combinatorics
through quasi-symmetric functions [\citet{stanley2}]. The descent theory
underlying shuffling makes equivalent appearances in the basic task of
carries when adding integers [Diaconis and Fulman (\citeyear{DFa,DFb,DFc})].

\section{Analysis of one pass through a shelf shuffler}\label{sec3}

This section gives a fairly complete analysis of a single pass through
a shelf shuffler. Section~\ref{newsec31} gives several equivalent descriptions
of the shuffle. In Section~\ref{sec31}, a closed-form formula for the chance
of any permutation $w$ is given. This in turn depends only on the
number of ``valleys'' in $w$. The number of permutations with $j$
valleys is easily calculated, and so exact computations for any of the
distances above are available. Section~\ref{newsec33} uses the exact formulas
to get asymptotic rates of convergence for $l_\infty$ and separation
distances. Section~\ref{sec32} gives the distribution of such
permutations by
cycle type. Section~\ref{sec33} gives the distribution of the ``shape'' of
such a permutation under the Robinson--Schensted--Knuth map.
Section~\ref{sec34} gives the distribution of the number of descents.
We find
it surprising that a real-world applied problem makes novel contact
with elegant combinatorics. In Section~\ref{sec4}, iterations of a shelf
shuffler are shown to be equivalent to shelf shuffling with more
shelves. Thus all of the formulas of this section apply.

\subsection{Alternative descriptions}\label{newsec31}

Consider two basic shelf shufflers: for the first, a~deck of $n$ cards
is sequentially distributed on one of $m$ shelves. (Here, $n=52,
m=10$, are possible choices.) Each time, the cards are taken from the
\textit{bottom} of the deck, a shelf is chosen at random from one to
$m$, and the bottom card is placed on top of any previous cards on the
shelf. At the end, the packets on the shelves are unloaded into a
final deck of $n$. This may be done in order or at random; it turns
out not to matter. \citet{bayer} called this an \textit{inverse
$m$-shuffle}.

The second shuffling scheme, that is the main object of the present
study, is based on
$m$ shelves. At each stage that a card is\vadjust{\goodbreak} placed on a shelf, the
choice of whether to put it on the top or the bottom of the existing
pile on that shelf is made at random ($1/2$ each side). This will be
called a \textit{shelf shuffle}. There are several equivalent
descriptions of shelf shuffles:

%
\begin{figure}[b]

\includegraphics{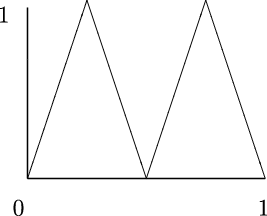}

\caption{Two shelves in shelf shuffle.}
\label{fig1}
\end{figure}

%
\begin{desc}[(Shelf shuffles)]\label{desc1}
A deck of cards is initially in order $1,2,3,\ldots,n$. Label the
back of each card with $n$ random numbers chosen at random between
one and $2m$. Remove all cards labeled 1 and place them on top,
keeping them in the same relative order. Then remove all cards
labeled 2 and place them under the cards labeled 1, reversing their
relative order. This continues with the cards labeled 3, labeled 4,
and so on, reversing the order in each even labeled packet. If at
any stage there are no cards with a given label, this empty packet
still counts in the alternating pattern.

For example, a twelve-card deck with $2m=4$,\vspace*{4pt}
\begin{center}
\begin{tabular}{lcccccccccccc}
Label&2&1&1&4&3&3&1&2&4&\phantom{0}3&\phantom{0}4&\phantom{0}1\\
Card&1&2&3&4&5&6&7&8&9&10&11&12
\end{tabular}\vspace*{4pt}
\end{center}
is reordered as\vspace*{4pt}
\begin{center}
\begin{tabular}{cccccccccccc}
2&3&7&12&8&1&5&6&10&11&9&4.
\end{tabular}
\end{center}
\end{desc}

%
\begin{desc}[(Inverse shelf shuffles)]\label{desc2}
Cut a deck of $n$ cards into $2m$ piles according to a multinomial
distribution; thus the number of cards cut off in pile $i$ has the
same distribution as the number of balls in the $i$th box if $n$
balls are dropped randomly into $2m$ boxes. Reverse the order of
the even-numbered packets. Finally, riffle shuffle the $2m$ packets
together by the Gilbert--Shannon--Reeds (GSR) distribution
\citet{bayer} dropping each card sequentially with probability
proportional to packet size. This makes all possible interleavings
equally likely.
\end{desc}

%
%

%
\begin{desc}[(Geometric description)]\label{desc3}
Consider the function $f_m(x)$ from $[0,1]$ to $[0,1]$ which has
``tents,'' each of slope $\pm2m$ centered at
$\frac1{2m},\frac3{2m},\break\frac5{2m},\ldots, \frac{2m-1}{2m}$.
Figure~\ref{fig1} illustrates an example with $m=2$. Place $n$ labeled
points uniformly at random into the unit interval. Label them, from
left to right, $x_1,x_2,\ldots,x_n$. Applying $f_m$ gives
$y_i=f_m(x_i)$. This gives the permutation
\[
\begin{tabular}{cccc} 1&2&$\cdots$&$n$
\\
$\pi_1$&$\pi_2$&$\cdots$&$\pi_n$
\end{tabular}
\]
with $\pi_1$ the relative position from the bottom of
$y_1,\ldots,\pi_i$ the relative position from the bottom of $y_i$ among
the other $y_j$. This permutation has the distribution of an inverse
shelf shuffle. It is important to note that the natural distances to
uniformity (total variation, separation, $l_\infty$) are the same for
inverse shuffles and forward shuffles. In Section~\ref{sec4}, this description
is used to show that repeated shelf shuffling results in a shelf
shuffle with more shelves.\looseness=1
\end{desc}

\subsection{Formula for the chance of a permutation produced by a shelf
shuffler}\label{sec31}

To describe the main result, we call $i$ a \textit{valley} of the
permutation $w\in S_n$ if $1<i<n$ and $w(i-1)>w(i)<w(i+1)$. Thus
$w=5{\underline{1}}367{\underline{2}}4$ has two valleys. The number of valleys is
classically used as a test of randomness for time series. See
\citet{warren} and their references. If $v(n,k)$ denotes the number of
permutations on $n$ symbols with $k$ valleys, then \citet{warren}
$v(1,0)=1, v(n,k)=(2k+2)v(n-1,k)+(n-2k)v(n-1,k-1)$. So $v(n,k)$ is
easy to compute for numbers of practical interest. Asymptotics are in
\citet{warren} which also shows the close connections between valleys
and descents.

%
\begin{thmm}\label{shelf}
The chance that a shelf shuffler with $m$ shelves and $n$ cards
outputs a permutation $w$ is
\[
\frac{4^{v(w)+1}}{2(2m)^n}\sum_{a=0}^{m-1}
\pmatrix{n+m-a-1
\cr
n}\pmatrix{n-1-2v(w)
\cr
a-v(w)},
\]
where $v(w)$ is the number of valleys of $w$. This can be seen to be
the coefficient of~$t^m$ in
\[
\frac{1}{2(2m)^n}\frac{(1+t)^{n+1}}{(1-t)^{n+1}} \biggl(\frac{4t}{(1+t)^2}
\biggr)^{v(w)+1}.
\]
\end{thmm}

%
\begin{defn*}
Suppose that $m=1$. Then the theorem yields the uniform distribution
on the $2^{n-1}$ permutations with no valleys;
a permutation with one or more valleys
occurs with probability 0.
Permutations with no
valleys are also sometimes called unimodal permutations. These arise
in social choice theory through Coombs's ``unfolding'' hypothesis
\citet{D88}, Chapter~6. They also appear in dynamical systems and magic tricks
see \citet{DiaconisGraham2011}, Chapter~5.
\end{defn*}

%
\begin{rem*}
By considering the cases $m\geq n$ and $n\geq m$ we see that, in the
formula of Theorem~\ref{shelf}, the range of summation can be taken
up to $n-1$ instead of $m-1$. This will be useful later.
\end{rem*}

Theorem~\ref{shelf} makes it easy to compute the distance to
stationarity for any of the metrics in Section~\ref{sec22}. Indeed, the
separation and $l_\infty$ distance is attained at either permutations
with a maximum number of valleys (when $n=52$, this maximum is~25) or
for permutations with 0 valleys. For the total variation distance,
with $P_m(v)$ denoting the probability in Theorem~\ref{shelf},
\[
\|P_m-U\|_{\tv}=\frac12\sum_{a=0}^{\lfloor{(n-1)}/2\rfloor}v(n,a)
\biggl\llvert P_m(a)-\frac1{n!}\biggr\rrvert.
\]

%
\begin{table}
\tabcolsep=0pt
\caption{Distances for various numbers of shelves $m$}\label{table1}
\begin{tabular*}{\textwidth}{@{\extracolsep{\fill}}lcccccccccccc@{}}
\hline
$\bolds{m}$&\textbf{10}&\textbf{15}&\textbf{20}&\textbf{25}&\textbf
{30}&\textbf{35}&\textbf{50}&\textbf{100}&\textbf{150}&\textbf
{200}&\textbf{250}&\textbf{300}\\
\hline
$\|P_m-U\|_{\tv
}$&1&0.943&0.720&0.544&0.391&0.299&0.159&0.041&0.018&0.010&0.007&0.005\\[2pt]
$\separ(P_m)$&1&1&1&1&1&0.996&0.910&0.431&0.219&0.130&0.085&0.060\\[2pt]
$\|P_m-U\|_\infty$&$\infty$&$\infty$&$\infty$&45\mbox
{,}118&3961&716&39&1.9&0.615&0.313&0.192&0.130\\
\hline
\end{tabular*}
\end{table}

Table~\ref{table1} gives these distances when $n=52$ for various
numbers of
shelves~$m$. Larger values of $m$ are of interest because of the
convolution results explained in Section~\ref{sec4}. These numbers show that
ten shelves are woefully insufficient. Indeed, 50 shelves are hardly
sufficient.

To prove Theorem~\ref{shelf}, we will relate it to the following
$2m$-shuffle on the hyperoctahedral group $B_n$: cut the deck
multinomially into $2m$ piles. Then flip over the odd numbered stacks,
and riffle the piles together, by dropping one card at a time from one
of the stacks (at each stage with probability proportional to stack
size). When $m=1$ this shuffle was studied in \citet{bayer}, and for
larger $m$ it was studied in \citet{F1}.

It will be helpful to have a description of the inverse of this
$2m$-shuffle. To each of the numbers $\{1,\ldots,n\}$ is assigned
independently and uniformly at random one of $-1,1,-2,2,\ldots,-m,m$.
Then a signed permutation is formed by starting with numbers mapped to
$-1$ (in decreasing order and with negative signs), continuing with
the numbers mapped to $1$ (in increasing order and with positive
signs), then continuing to the numbers mapped to $-2$ (in decreasing
order and with negative signs), and so on. For example the assignment
{\renewcommand{\theequation}{$*$}
\begin{eqnarray}\label{eqZ}\quad
\{1,3,8\}&\mapsto&-1,\qquad\{5\}\mapsto1,\qquad\{2,7\}\mapsto2,
\nonumber
\\[-8pt]
\\[-8pt]
\nonumber
\{6\}&\mapsto&-3,\qquad\{
4\}
\mapsto3
\end{eqnarray}}
\hspace*{-2pt}leads to the signed permutation
{\renewcommand{\theequation}{$**$}
\begin{equation}\label{eqZZ}
\matrix{
-8 & -3&  -1& 5& 2 &7 &-6 &4 }.
\end{equation}}

The proof of Theorem~\ref{shelf} depends on an interesting relation
with shuffles for signed permutations (hyperoctahedral group). This is
given next followed by the proof of Theorem~\ref{shelf}.

Theorem~\ref{peaks} gives a formula for the probability for $w$ after
a hyperoctahedral $2m$-shuffle, when one forgets signs. Here $p(w)$ is
the number of peaks of~$w$, where~$i$ is said to be a peak of $w$ if
$1<i<n$ and $w(i-1)<w(i)>w(i+1)$. Also $\Lambda(w)$ denotes the peak
set of $w$ and $D(w)$ denotes the descent set of $w$ [i.e., the set of
points $i$ such that $w(i)>w(i+1)$]. Finally, let $[n] =\{1,\ldots,n\}$.

%
\begin{thmm}\label{peaks}
The chance of a permutation $w$ obtained by performing a $2m$ shuffle
on the hyperoctahedral group and then forgetting signs is
\[
\frac{4^{p(w^{-1})+1}}{2(2m)^n}\sum_{a=0}^{m-1}
\pmatrix{n+m-a-1
\cr
n}\pmatrix{n-1-2p\bigl(w^{-1}\bigr)
\cr
a-p
\bigl(w^{-1}\bigr)},
\]
where $p(w^{-1})$ is the number of peaks of $w^{-1}$.
\end{thmm}

\begin{pf}
Let $P'(m)$ denote the set of nonzero integers of absolute value at
most~$m$, totally ordered so that
\[
-1 \prec1 \prec-2 \prec2 \prec\cdots\prec-m \prec m .
\]
Then given a permutation $w=(w_1,\ldots,w_n)$, page 768 of \citet{stem}
defines a quantity $\Delta(w)$. (Stembridge calls it
$\Delta(w,\gamma)$, but throughout we always choose $\gamma$ to be the
identity map on $[n]$, and so suppress the symbol $\gamma$ whenever he
uses it.) By definition, $\Delta(w)$ enumerates the number of maps $f:
[n] \mapsto P'(m)$ such that:
\begin{itemize}
\item$f(w_1)\preceq\cdots\preceq f(w_n)$;
\item$f(w_i)=f(w_{i+1})>0 \Rightarrow i\notin D(w)$;
\item$f(w_i)=f(w_{i+1})<0 \Rightarrow i\in D(w)$.
\end{itemize}

We claim that the number of maps $f\dvtx[n]\mapsto P'(m)$ with the above
three properties is equal to $(2m)^n$ multiplied by the chance that a
hyperoctahedral $2m$-shuffle results in the permutation $w^{-1}$. This
is most clearly explained
continuing example~(\ref{eqZ}), (\ref{eqZZ}) above let $w,f$ be
\[
\begin{tabular}{lcccccccccr} $w  =  \phantom{-}8$ & \phantom{$-$}3 & \phantom{$-$}1 & 5 & 2 & 7 & \phantom{$-$}6 & 4,
\\
$f  =  -1$ & $-1$ & $-1$ & 1 & 2 & 2 & $-3$ & 3.
\end{tabular}
\]
Here, $f$ is monotone read left to right, $f(8) = f(3) = f(1) = -1$
corresponds to the descents in the first two positions and
$f(2) = f(7) = 2$ corresponds to the ascent.
This $f$ arises from the
description of the inverse hyperoctahedral $2m$-shuffle
in (\ref{eqZ}), (\ref{eqZZ}) above, the assignment yields $w$. This proves the
claim.\vadjust{\goodbreak}

Let $\Lambda(w)$ denote the set of peaks of $w$. From Proposition 3.5
of \citet{stem},
\[
\Delta(w) = 2^{p(w)+1} \sum_{E \subseteq[n-1]\dvtx\Lambda(w)
\subseteq E \triangle(E+1)}
L_E.
\]
Here $E+1$ means the elements of $E$ with one added to each,
\[
L_E=\mathop{\sum_{1\leq i_1\leq\cdots\leq i_n\leq m}}_{k\in
E\Rightarrow i_k<i_{k+1}}1.
\]
Here $L_E$ is the number of ordered $n$-tuples $(i_1,i_2,\ldots,i_n)$
of integers between $1$ and $m$ which are nondecreasing, and strictly
increasing at positions
$k$ in~$E$.
$\triangle$~denotes symmetric difference, that is, $A\triangle
B=(A-B)\cup(B-A)$. Now a simple combinatorial argument shows that
$L_E=\bigl({n+m-|E|-1\atop n}\bigr)$. Indeed, $L_E$ is equal to the number of
integral $i_1,\ldots,i_n$ with $1\leq i_1\leq\cdots\leq i_n\leq m-|E|$,
which by a ``stars and bars'' argument is $\bigl({n+m-|E|-1\atop n}\bigr)$.
Thus
\[
\Delta(w) = 2^{p(w)+1} \sum_{E \subseteq[n-1]\dvtx\Lambda(w)
\subseteq E \triangle(E+1)}
\pmatrix{n+m-|E|-1
\cr
n}.
\]
Now let us count the number of $E$ of size $a$ appearing in this sum.
For each $j\in\Lambda(w)$, exactly one of $j$ or $j-1$ must belong to
$E$, and the remaining $n-1-2p(w)$ elements of $[n-1]$ can be
independently and arbitrarily included in $E$. Thus the number of sets
$E$ of size $a$ appearing in the sum is
$2^{p(w)}\bigl({n-1-2p(w)\atop a-p(w)}\bigr)$. Hence
\[
\Delta(w)=\frac{4^{p(w)+1}}2\sum_{a=0}^{n-1}
\pmatrix{n+m-a-1
\cr
n}\pmatrix{n-1-2p(w)
\cr
a-p(w)},
\]
which completes the proof.
\end{pf}

\begin{pf*}{Proof of Theorem~\ref{shelf}} To deduce Theorem
\ref{shelf} from Theorem~\ref{peaks}, we see that a
shelf shuffle with $m$ shelves is equivalent to taking $w'$ to be
the inverse of a permutation after a hyperoctahedral $2m$-shuffle
(forgetting about signs),
then taking a permutation $w$ defined by $w(i)=n-w'(i)+1$. Thus the
shelf shuffle formula is obtained from the hyperoctahedral
$2m$-shuffle formula by replacing peaks by valleys.
\end{pf*}

%
\begin{rems*}

\begin{itemize}
\item The paper \citet{F1} gives an explicit formula for the chance of
a signed permutation after a $2m$-shuffle on $B_n$ in terms of
cyclic descents. Namely it shows this probability to be
\[
\frac{\bigl({m+n-cd(w^{-1})\atop n}\bigr)}{(2m)^n},\vadjust{\goodbreak}
\]
where $cd(w)$ is the number of cyclic descents of $w$, defined as
follows: Ordering the integers $1<2<3<\cdots<\cdots<-3<-2<-1$:
\begin{itemize}[$-$]

\item[$-$]$w$ has a cyclic descent at position $i$ for $1 \leq i \leq n-1$ if
$w(i)>w(i+1)$.
\item[$-$]$w$ has a cyclic descent at position $n$ if $w(n)<0$.
\item[$-$]$w$ has a cyclic descent at position $1$ if $w(1)>0$.
\end{itemize}
For example the permutation ${3\enskip 1 \enskip -2 \enskip 4\enskip 5 }$ has two cyclic
descents at position~1 and a cyclic descent at position 3, so
$cd(w)=3$.

This allows one to study aspects of shelf shufflers by lifting the
problem to~$B_n$, using cyclic descents (where calculations are often
easier), then forgetting about signs. This idea was used in \citet{F1}
to study the cycle structure of unimodal permutations, and in
\citet{aguiar04} to study peak algebras of types $B$ and $D$.

The idea of lifting the problem to type $B_n$ leads to a total variation
upper bound. Indeed, from the proof of Theorem~\ref{shelf} the total variation
distance after a shelf-shuffler with $m$ shelves to uniform is equal to the
total variation distance of a hyperoctahedral $2m$ shuffle to uniform, after
one forgets about signs. Now from \citet{bayer} or \citet{F1}, the
total variation distance
of a hyperoctahedral $2m$ shuffle to uniform, when one does not forget about
signs, is equal to the total variation distance of an ordinary $m$ riffle
shuffle to uniform on the symmetric group---a quantity thoroughly
studied in
\citet{bayer}. Thus the total variation distance after a shelf-shuffler
with $m$ shelves
to uniform is at most the total variation distance of an ordinary $m$ riffle
shuffle to uniform on the symmetric group.

\item The appearance of peaks in the study of shelf shufflers is
interesting, as peak algebras have appeared in various parts of
mathematics. \citet{nyman} proves that the peak algebra is a
subalgebra of the symmetric group algebra, and connections with
geometry of polytopes can be found in \citet{aguiar06} and
\citet{billera}. There are also close connections with the theory of
$P$-partitions [\citeauthor{P1} (\citeyear{P1,P2}), \citet{stem}].
\end{itemize}
\end{rems*}

The following corollary shows that for a shelf shuffler of $n$ cards
with $m$ shelves, the chance of a permutation $w$ with $v$ valleys is
monotone decreasing in $v$. Thus, the identity (or any other unimodal
permutation) is most likely and an alternating permutation$\ldots$ (down, up,
down, up$,\ldots$) is least likely. From Theorem~\ref{shelf},
the chance of a fixed permutation with $v$ valleys is
%
\setcounter{equation}{0}
\begin{equation}
P(v)=\frac{4^{v+1}}{2(2m)^n}\sum_{a=0}^{n-1}
\pmatrix{n+m-1-a
\cr
n}\pmatrix{n-1-2v
\cr
a-v}. \label{31}
\end{equation}

%
\begin{cor}
For $P(v)$ defined at \eqref{31}, $P(v)\geq P(v+1), 0\leq
v\leq(n-1)/2$.
\label{cor31}
\end{cor}

\begin{pf}
Canceling common terms, and setting $a-v=j$ (so $a=j+v$) in~\eqref{31}, we have
$\frac{2(2m)^n}{4^{v+1}}P(v)=\sum_{j=0}^{n-1-2v} f(j+v)\bigl({n-1-2v\atop
j}\bigr)=2^{n-1-2v}\times E(f(S_{n-1-2v}+v))$
with $f(a)=\bigl({n+m-1-a\atop n}\bigr)$ and $S_{n-1-2v}$ distributed as
binomial $(n-1-2v,\frac12)$. The proposed inequality is equivalent to
%
\begin{equation}
E \bigl(f (S_{n-1-2v}+v ) \bigr)\geq E \bigl(f (S_{n-1-2v-2}+v+1 )
\bigr). \label{32}
\end{equation}
To prove this, represent $S_{n-1-2v}=S_{n-1-2v-2}+Y_1+Y_2$, with $Y_i$
independent taking values in $\{0,1\}$, with probability $1/2$. Then
\eqref{32} is equivalent to
%
\begin{eqnarray}\label{33}
&&\sum_j \biggl[\frac14f(j+v)+\frac12f(j+v+1)+
\frac14f(j+v+2)-f(j+v+1) \biggr]
\nonumber
\\[-8pt]
\\[-8pt]
\nonumber
&&\qquad{}\times P\{S_{n-1-2v-2}=j\}\geq0.
\end{eqnarray}
Thus if $\frac12f(j+v)+\frac12f(j+v+2)\geq f(j+v+1)$, for example,
$f(a)$ is
convex, we are done. Writing out the expression
$f(a)+f(a+2)\geq2f(a+1)$ and canceling common terms, it must be shown
that
%
\begin{eqnarray}
&&(m+n-1-a) (m+n-2-a)+(m-1-a) (m-2-a)
\nonumber
\\[-8pt]
\\[-8pt]
\nonumber
&&\qquad\geq2(m+n-2-a) (m-1-a) \label{34}
\end{eqnarray}
for all $0\leq a\leq n-1$. Subtracting the right-hand side from the left,
the coefficients of $a^2$ and $a$ cancel, leaving $n(n-1)\geq0$.
\end{pf}

\subsection{\texorpdfstring{Asymptotics for the $\|P-U\|_\infty$ and separation distances}
{Asymptotics for the ||P-U||infinity and separation distances}}\label{newsec33}

Recall the distances
$\|P-U\|_\infty=\max_w\llvert1-\frac{P(w)}{U(w)}\rrvert$ and
$\separ(P)=\max_w (1-\frac{P(w)}{U(w)} )$.

%
\begin{thmm}\label{thm34}
Consider the shelf shuffling measure $P_m$ with $n$ cards and $m$
shelves. Suppose that $m=cn^{3/2}$. Then, as $n$ tends to infinity
with $0<c<\infty$ fixed,
\begin{eqnarray*}
\|P_m-U\|_\infty&\sim& e^{1/(12c^2)}-1,
\\
\separ(P_m)&\sim& 1-e^{-1/(24c^2)}.
\end{eqnarray*}
\end{thmm}

%
\begin{rem*}
We find it surprising that this many shelves are needed. For
example, when $n=52$, to make the distance less than $1/100$,
$m\doteq1085$ shelves are required for $\|P_m-U\|_\infty$ and
$m\doteq764$ are required for $\separ(P_m)$. The order $m\doteq n^{{3}/{2}}$
in Theorem~\ref{thm34} can be understood as follows:
\citet{bayer} show that it takes $k\doteq\frac{3}{2} \log_2{n}$
riffle shuffles to mix $n$ cards in total variation. Now $k$
riffle shuffles correspond to $2^k=n^{{3}/{2}}$ shelves. Of course,
the $\frac{3}{2}$ in riffle shuffling is the result of a careful computation.
\end{rem*}

\begin{pf}
Using Corollary~\ref{cor31}, the distance is achieved at the
identity permutation or a permutation with\vadjust{\goodbreak} $\lfloor(n-1)/2\rfloor$
valleys. For the identity, consider $n!P_m(\mathrm{id})$. Using
Theorem~\ref{shelf},
%
\begin{equation}
n!P_m(\mathrm{id})=\frac{2(n!)}{(2m)^n}\sum
_{a=0}^{n-1}\pmatrix{m+n-a-1
\cr
n}\pmatrix{n-1
\cr
a}.
\label{35}
\end{equation}
To bound this sum, observe that $\bigl({n-1\atop a}\bigr)/2^{n-1}$ is the
binomial probability density. To keep the bookkeeping simple, assume
throughout that $n$ is odd. The argument for even $n$ is similar.

For $a=\frac{n-1}2+j$, the local central limit theorem as in
\citeauthor{feller} [(\citeyear{feller}), Chapter~VII.2], shows
%
\begin{equation}
\frac{\bigl({n-1\atop {(n-1)}/2+j}\bigr)}{2^{n-1}}\sim\frac
{e^{-2j^2/n}}{\sqrt{\pi n/2}}\qquad \mbox{for }j=o\bigl(n^{2/3}
\bigr). \label{36}
\end{equation}
In the following, we show further that
%
\begin{eqnarray}\label{37}
\frac{n!}{m^n}\pmatrix{m+{(n-1)}/2-j
\cr
n}\sim e^{-1/{(24c^2)}+{j}/{(c\sqrt{n})}}
\nonumber
\\[-8pt]
\\[-8pt]
\eqntext{\mbox{uniformly for }j=o(n).}
\end{eqnarray}
Combining \eqref{36}, \eqref{37}, gives a Riemann sum for the integral
\[
\frac{e^{-1/{(24c^2)}}}{\sqrt{\pi/2}}\int_{-\infty}^\infty e^{-2x^2+x/c}
\,dx=e^{1/(12c^2)},
\]
the claimed result. This part of the argument follows
\citeauthor{feller} [(\citeyear{feller}), Chapter~VII.2], and we suppress further details. To
complete the argument the tails of the sum in~\eqref{35} must be
bounded.

We first prove \eqref{37}. From the definitions
\[
\frac{n!}{m^n}\pmatrix{m-j+{(n-1)}/2
\cr
n}=\prod
_{i=-{(n-1)}/2}^{{(n-1)}/2} \biggl(1-\frac{j}{m}+
\frac{i}{m} \biggr)
\]
using $\log(1-x)=-x-\frac{x^2}2+O(x^3)$,
%
\begin{eqnarray}
\label{38}&& \sum_{i=-{(n-1)}/2}^{{(n-1)}/2}\log\biggl(1-
\frac{j}{m}+\frac{i}{m} \biggr)\nonumber\\
 &&\qquad=-\sum
_i \biggl(-\frac{j}{m}+\frac{i}{m} \biggr)-
\frac12\sum_i \biggl(-\frac{j}{m}+
\frac{i}{m} \biggr)^2 +nO \biggl( \biggl(\frac{n}{m}
\biggr)^3 \biggr)
\nonumber
\\[-8pt]
\\[-8pt]
\nonumber
&&\qquad=\frac{nj}{m}-\frac12 \biggl(\frac{nj^2}{m^2}+\frac1{12}
\frac{n(n^2-1)}{m^2} \biggr)+O \biggl(\frac1{\sqrt{n}} \biggr)
\\
&&\qquad=\frac{j}{c\sqrt{n}}-\frac{j^2}{2c^2n^2}-\frac1{24c^2}+O \biggl(
\frac1{\sqrt{n}} \biggr).\nonumber
\end{eqnarray}
The error term in \eqref{38} is uniform in $j$. For $j=o(n),
j^2/n^2=o(1)$ and \eqref{37} follows.

To bound the tails of the sum, first observe that \eqref{38} implies
that\break
$\frac{n!}{m^n}\bigl({m-j+{(n-1)}/2\atop n}\bigr)=e^{O (\sqrt{n} )}$
for all $j$. From Bernstein's inequality, if $X_i=\pm1$ with
probability $1/2$, $P(|X_1+\cdots+X_{n-1}|>a)\leq2e^{-a^2/(n-1)}$. Using
this, the sum over $|j|\geq An^{3/4}$ is negligible for $A$
sufficiently large.

The Gaussian approximation to the binomial works for $j\ll n^{2/3}$.
To bound the sum for $|j|$ between $n^{2/3}$ and $n^{3/4}$, observe
from \eqref{38} that in this range,
$\frac{n!}{m^n}\bigl({m-j+{(n-1)}/2\atop n}\bigr)=O (e^{n^{1/4}} )$.
Then \citeauthor{feller} [(\citeyear{feller}), page~195] shows
\[
\frac{\pmatrix{n-1\cr{(n-1)}/2+j}}{2^{n-1}}\sim\frac1{\sqrt{\pi
n/2}}e^{-(1/2)(j)^2/(n/4)-f (j/\sqrt{n/4} )}
\]
with $f(x)=\sum_{a=3}^\infty\frac{ (1/2 )^{a-1}+
(-1/2 )^{a-1}}{a(a-1)} (\frac1{\sqrt{n/4}} )^{a-2}x^a=
c_1\frac{x^4}{n}+c_2\frac{x^6}{n^2}+\cdots$ for explicit constants
$c_1,c_2,\ldots.$ For $\theta_1n^{2/3}\leq|j|\leq\theta_2n^{3/4}$, the
sum under study is dominated by $A\sum_{j\geq n^{2/3}}e^{-Bj^{1/6}}$
which tends to zero.

The separation distance is achieved at permutations with $\frac{n-1}2$
valleys (recall we are assuming that $n$ is odd). From \eqref{31},
\[
1-n!P_m \biggl(\frac{n-1}2 \biggr)=1-\frac{n!}{m^n}
\pmatrix{m+{(n-1)}/2
\cr
n}.
\]
The result now follows from \eqref{37} with $j=0$.
\end{pf}

%
\begin{rem*}
A similar argument allows asymptotic evaluation of total variation.
We have not carried out the details.
\end{rem*}

\subsection{Distribution of cycle type}\label{sec32}

The number of fixed points and the number of cycles are classic
descriptive statistics of a permutation. More generally, the number of
$i$-cycles for $1\leq i\leq n$ has been intensively studied
[\citet{shepp,DMP}]. This section investigates the distribution of cycle
type of a permutation $w$ produced from a shelf shuffler with $m$
shelves and $n$ cards. Similar results for ordinary riffle shuffles
appeared in Diaconis, \mbox{McGrath} and Pitman (\citeyear{DMP}), and closely related results in the type $B$
case (not in the language of shelf-shuffling) appear in \citeauthor{F1} (\citeyear{F1,F2}).
Recall also that in the case of one shelf, the shelf shuffler
generates one of the $2^{n-1}$ unimodal permutations uniformly at random.
The cycle structure of unimodal permutations has been studied in
several papers in the literature: see \citeauthor{F1} (\citeyear{F1,F2}), \citet{thibon} for
algebraic/combinatorial approaches and \citet{gannon,rogers} for
approaches using dynamical systems.

For what follows, we define
\[
f_{i,m}=\frac1{2i}\mathop{\sum_{d|i}}_{d\ \mathrm{odd}}
\mu(d) (2m)^{i/d},
\]
where $\mu$ is the M\"obius function of elementary number theory:
$\mu(d)=(-1)^k$ if $d$ is a square free number with $k$ prime factors,
$\mu(1)=1$ and $\mu(d)=0$ otherwise.

%
\begin{thmm}\label{cycle}
Let $P_m(w)$ denote the probability that a shelf shuffler with $m$
shelves produces a permutation $w$. Let $N_i(w)$ denote the number
of $i$-cycles of a permutation $w$ in $S_n$. Then
%
\begin{equation}
1+\sum_{n \geq1} u^n \sum
_{w \in S_n} P_{m}(w) \prod_{i \geq1}x_i^{N_i(w)}
= \prod_{i\geq1} \biggl( \frac{1+x_i (u/2m)^i}{1-x_i (u/2m)^i}
\biggr)^{f_{i,m}}. \label{cycform}
\end{equation}
\end{thmm}

\begin{pf}
By the proof of Theorem~\ref{shelf}, a permutation produced by a
shelf shuffler with $m$ shelves is equivalent to forgetting signs
after the inverse of a type $B$ riffle shuffle with $2m$ piles, then
conjugating by the longest element $n,n-1,\ldots,1$. Since a
permutation and its inverse have the same cycle type and conjugation
leaves cycle type invariant, the result follows from either
\citeauthor{F1} [(\citeyear{F1}), Theorem~7] or \citeauthor{F2} [(\citeyear{F2}), Theorem~9] both of which derived the
generating function for cycle type after type $B$ shuffles.
\end{pf}

Theorem~\ref{cycle} leads to several corollaries. We say that a random
variable $X$ is binomial $(n,p)$ if
$\pp(X=j)=\bigl({n\atop j}\bigr)p^j(1-p)^{n-j}, 0\leq j\leq n$, and that $X$ is
negative binomial with parameters $(f,p)$ if
$\pp(X=j)=\bigl({f+j-1\atop j}\bigr)p^j(1-p)^f, 0\leq j<\infty$.
As usual,
the products in the generating function on the right of (\ref{cycform}) 
correspond to the convolution
of the corresponding measures.

%
\begin{cor}\label{asymptot}
Let $N_i(w)$ be the number of $i$-cycles of a permutation~$w$.
\begin{longlist}[(1)]

\item[(1)] Fix $u$ such that $0<u<1$. Then choose a random number $N$ of
cards so that $\pp(N=n)=(1-u)u^n$. Let $w$ be produced by a shelf
shuffler with $m$ shelves and $N$ cards. Then any finite number of
the random variables $\{N_i\}$ are independent, and $N_i$ is
distributed as the convolution of a
binomial $ (f_{i,m},\frac{(u/2m)^i}{1+(u/2m)^i} )$ and a
negative binomial with parameters $(f_{i,m},(u/2m)^i)$.

\item[(2)] Let $w$ be produced by a shelf shuffler with $m$ shelves and $n$
cards. Then in the $n\to\infty$ limit, any finite number of the
random variables $\{N_i\}$ are independent. The $N_i$ are
distributed as the convolution of a
binomial $ (f_{i,m},\frac1{(2m)^i+1} )$ and a negative
binomial with parameters $(f_{i,m},(1/2m)^i)$.
\end{longlist}
\end{cor}

\begin{pf}
Setting all $x_i=1$ in equation \eqref{cycform} yields the equation
%
\begin{equation}
(1-u)^{-1}=\prod_{i\geq1} \biggl(
\frac{1+(u/2m)^i}{1-(u/2m)^i} \biggr)^{f_{i,m}}. \label{set1}
\end{equation}
Taking reciprocals of equation \eqref{set1} and multiplying by
equation \eqref{cycform} gives the equality
%
\begin{eqnarray}\label{arr}
&&(1-u)+\sum_{n\geq1}(1-u)u^n\sum
_{w\in S_n}P_{m}(w)\prod_{i\geq1}x_i^{n_i(w)}
\nonumber
\\[-8pt]
\\[-8pt]
\nonumber
&&\qquad=\prod_{i\geq1} \biggl(\frac{1+x_i(u/2m)^i}{1+(u/2m)^i}
\biggr)^{f_{i,m}}\cdot\biggl(\frac{1-(u/2m)^i}{1-x_i(u/2m)^i} \biggr
)^{f_{i,m}}.
\end{eqnarray}
This proves part 1 of the theorem, the first term on the right
corresponding to the convolution of binomials, and the second term to
the convolution of negative binomials.

The second part follows from the claim that if a generating function
$f(u)$ has a Taylor series which converges at $u=1$, then the $n \to
\infty$ limit of the coefficient of $u^n$ in $f(u)/(1-u)$ is $f(1)$.
Indeed, write the Taylor expansion $f(u)=\sum_{n=0}^{\infty} a_n u^n$
and observe that the coefficient of $u^n$ in $f(u)/(1-u)$ is
$\sum_{i=0}^n a_i$. Now apply the claim to equation \eqref{arr} with
all but finitely many $x_i$ equal to $1$.
\end{pf}

%
\begin{rem*}
For example, when $i=1, f_{i,m}=m$; the number of fixed points are
distributed as a sum of binomial $ (m,\frac1{2m+1} )$ and
negative binomial$ (m,\frac1{2m} )$. Each of these
converges to $\operatorname{Poisson}(1/2)$ and so the number of fixed points is
approximately $\operatorname{Poisson}(1)$. A similar analysis holds for the other
cycle counts. Corollary~\ref{asymptot} could also be proved by the
method of moments, along the lines of the arguments of \citet{DMP}
for the case of ordinary riffle shuffles.
\end{rem*}

For the next result, recall that the limiting distribution of the
large cycles of a uniformly chosen permutation in $S_n$ has been
determined by \citeauthor{go1} (\citeyear{go1,go2}), \citet{shepp},
\citeauthor{V12} (\citeyear{V12,V22}), and others. For instance the average length of the
longest cycle $L_1$ is approximately $0.63n$ and $L_1/n$ has a known
limiting distribution. The next result shows that even with a fixed
number of shelves, the distribution of the large cycles approaches
that of a uniform random permutation, as long as the number of cards
is growing. We omit the proof, which goes exactly along the lines of
the corresponding result for riffle shuffles in \citet{DMP}.

%
\begin{cor}
Fix $k$ and let $L_1(w), L_2(w),\ldots,L_k(w)$ be the lengths of the
$k$ longest cycles of $w \in S_n$ produced by\vadjust{\goodbreak} a shelf shuffler with
$m$ shelves. Then for $m$ fixed, or growing with $n$, as
$n\to\infty$,
\[
\bigl\llvert P_m\{L_1/n\leq t_1,
\ldots,L_k/n\leq t_n\}-P_\infty
\{L_1/n\leq t_1,\ldots,L_k/n\leq
t_n\}\bigr\rrvert\to0
\]
uniformly in $t_1,t_2,\ldots,t_k$.
\label{val}
\end{cor}

As a final corollary, we note that Theorems~\ref{shelf} and
\ref{cycle} give the following generating function for the joint
distribution of permutations by valleys and cycle type. Note that this
gives the joint generating function for the distribution of
permutations by peaks and cycle type, since conjugating by the
permutation $n,n-1,\ldots,1$ preserves the cycle type and swaps valleys
and peaks.

%
\begin{cor}\label{joint}
Let $v(w)$ denote the number of valleys of a permutation $w$. Then
\begin{eqnarray*}
&&\frac{t}{1-t}+\sum_{n\geq1}u^n\sum
_{w\in S_n}\frac12\frac{(1+t)^{n+1}}{(1-t)^{n+1}} \biggl(
\frac{4t}{(1+t)^2} \biggr)^{v(w)+1} \prod_{i\geq1}x_i^{N_i(w)}
\\
&&\qquad=\sum_{m\geq1}t^m\prod
_{i\geq1}\biggl(\frac{1+x_iu^i}{1-x_i u^i}\biggr)^{f_{i,m}}.
\end{eqnarray*}
The same result holds with $v(w)$ replaced by $p(w)$, the number of
peaks of~$w$.
\end{cor}

%
\begin{rem*}
There is a large literature on the joint distribution of
permutations by cycles and descents
[\citet{gessel,DMP,reiner,F3,blessen,poirier}] and by cycles and cyclic
descents \citeauthor{F1} (\citeyear{F4,F1,F2}), but Corollary~\ref{joint} seems to be the
first result on the joint distribution by cycles and peaks.
\end{rem*}

\subsection{Distribution of RSK shape}\label{sec33}

In this section we obtain the distribution of the
Robinson--Schensted--Knuth (RSK) shape of a permutation $w$ produced
from a shelf shuffler with $m$ shelves and $n$ cards. For background
on the RSK algorithm, see \citet{stanley}. The RSK bijection associates
to a permutation $w \in S_n$ a pair of standard Young tableaux
$(P(w),Q(w))$ of the same shape and size $n$. $Q(w)$ is called the
recording tableau of $w$.

To state our main result, we use a symmetric function $S_{\lambda}$
studied in \citet{stem} [a special case of the extended Schur functions
in \citet{kerov}]. One definition of the $S_{\lambda}$ is as the
determinant
\[
S_\lambda(y)=\det(q_{\lambda_i-i+j}),
\]
where $q_{-r}=0$ for $r>0$ and for $r \geq0$, $q_r$ is defined by
setting
\[
\sum_{n\geq0}q_nt^n=\prod
_{i\geq1}\frac{1+y_it}{1-y_it}.
\]
We also let $f_{\lambda}$ denote the number of standard Young tableaux
of shape $\lambda$.\vadjust{\goodbreak}

%
\begin{thmm}%
\label{shape}
The probability that a shelf shuffler with $m$ shelves and $n$ cards
produces a permutation with recording tableau $T$ is equal to
\[
\frac1{2^n}S_\lambda\biggl(\frac1{m},\ldots,\frac1{m}
\biggr)
\]
for any $T$ of shape $\lambda$, where $S_{\lambda}$ has $m$ variables.
Thus the probability that $w$ has RSK shape $\lambda$ is equal to
\[
\frac{f_\lambda}{2^n}S_\lambda\biggl(\frac1{m},\ldots,\frac1{m} \biggr).
\]

\end{thmm}

\begin{pf}
By the proof of Theorem~\ref{shelf}, a permutation produced by a
shelf shuffler with $m$ shelves is equivalent to forgetting signs
after the inverse of a type $B$ $2m$-shuffle, and then conjugating
by the permutation $n, n-1, \ldots,1$. Since a permutation and its
inverse have the same RSK shape [\citet{stanley}, Section~7.13], and
conjugation by $n,n-1,\ldots,1$ leaves the RSK shape unchanged
[\citet{stanley}, Theorem~A1.2.10], the result follows from
\citeauthor{F2} [(\citeyear{F2}), Theorem 8], who studied RSK shape after type $B$ riffle shuffles.
\end{pf}

\subsection{Distribution of descents}\label{sec34}
A permutation $w$ is said to have a descent at position $i$ $(1 \leq i
\leq n-1)$ if $w(i)>w(i+1)$. We let $d(w)$ denote the total number of
descents of $w$. For example the permutation $3\enskip 1\enskip 5\enskip 4\enskip 2$ has
$d(w)=3$ and descent set~$1,3,4$. The purpose of this section is to
derive a generating function for the number of descents in a
permutation $w$ produced by a shelf shuffler with $m$ shelves and $n$
cards. More precisely, we prove the following result.

%
\begin{thmm}\label{desgen}
Let $P_m(w)$ denote the probability that a shelf shuffler with $m$
shelves and $n$ cards produces a permutation $w$. Letting $[u^n]
f(u)$ denote the coefficient of $u^n$ in a power series $f(u)$, one
has that
%
\begin{equation}
\sum_{w\in S_n}P_m(w)t^{d(w)+1}=
\frac{(1-t)^{n+1}}{2^n}\sum_{k\geq1}t^k
\bigl[u^n\bigr]\frac{(1+u/m)^{km}}{(1-u/m)^{km}}. \label{des}
\end{equation}
\end{thmm}

The proof uses the result about RSK shape mentioned in Section~\ref{sec33},
and symmetric function theory; background on these topics can be found
in the texts by \citet{stanley} and \citet{mac}, respectively.

\begin{pf}
Let $w$ be a permutation produced by a shelf shuffler with $m$
shelves and $n$ cards. The RSK correspondence associates to $w$ a
pair of standard Young tableaux $(P(w),Q(w))$ of the same shape.
Moreover, there is a notion of descent set for standard Young
tableaux, and by Lemma 7.23.1 of \citet{stanley}, the descent set of
$w$ is equal to the descent set of $Q(w)$. Let $f_{\lambda}(r)$
denote the number of standard Young tableaux of shape $\lambda$ with\vadjust{\goodbreak}
$r$ descents. Then Theorem~\ref{shape} implies that
\[
\pp\bigl(d(w)=r\bigr)=\sum_{|\lambda|=n}
\frac{f_{\lambda}(r)}{2^n}S_\lambda\biggl(\frac1{m},\ldots,\frac1{m}
\biggr).
\]

By equation (7.96) of \citet{stanley}, one has that
\[
\sum_{r\geq0}f_\lambda(r)t^{r+1}=(1-t)^{n+1}
\sum_{k\geq1}s_\lambda(1,\ldots,1)t^k,
\]
where in the $k$th summand, $s_{\lambda}(1,\ldots,1)$ denotes the Schur
function with $k$ variables specialized to 1. Thus
\begin{eqnarray*}
&&\sum_{r\geq0}\pp\bigl(d(w)=r\bigr)\cdot
t^{r+1}
\\
&&\qquad=\sum_{r\geq0}\sum
_{|\lambda|=n}\frac{f_\lambda(r)}{2^n}S_\lambda\biggl(
\frac1{m},\ldots,\frac1{m} \biggr)\cdot t^{r+1}
\\
&&\qquad=\frac{(1-t)^{n+1}}{2^n}\sum_{k\geq1}t^k\sum
_{|\lambda|=n}S_\lambda\biggl(\frac1{m},\ldots,
\frac1{m} \biggr)s_\lambda(1,\ldots,1)
\\
&&\qquad=\frac{(1-t)^{n+1}}{2^n}\sum_{k\geq1}t^k
\bigl[u^n\bigr]\sum_{n\geq0}\sum
_{|\lambda|=n}S_\lambda\biggl(\frac1{m},\ldots,\frac1{m}
\biggr)s_\lambda(1,\ldots,1)\cdot u^n.
\end{eqnarray*}
From Appendix A.4 of \citet{stem}, if $\lambda$ ranges over all
partitions of all natural numbers, then
\[
\sum_\lambda s_\lambda(x)S_\lambda(y)=
\prod_{i,j\geq1}\frac{1+x_i y_j}{1-x_i y_j}.
\]
Setting $x_1=\cdots=x_k=u$ and $y_1=\cdots=y_m=\frac1{m}$ completes
the proof of the theorem.
\end{pf}

For what follows we let $A_n(t)=\sum_{w \in S_n} t^{d(w)+1}$ be the
generating function of elements in $S_n$ by descents. This is known as
the Eulerian polynomial, and from page 245 of \citet{comtet}, one has
that
%
\begin{equation}
A_n(t)=(1-t)^{n+1}\sum_{k\geq1}t^kk^n.
\label{eul}
\end{equation}
This also follows by letting $m\to\infty$ in equation \eqref{des}.

The following corollary derives the mean and variance of the number of
descents of a permutation produced by a shelf shuffler.

%
\begin{cor}\label{descor}
Let $w$ be a permutation produced by a shelf shuffler with $m$
shelves and $n \geq2$ cards.
\begin{longlist}[(1)]
\item[(1)] The expected value of $d(w)$ is $\frac{n-1}{2}$.
\item[(2)] The variance of $d(w)$ is $\frac{n+1}{12}+\frac{n-2}{6m^2}$.
\end{longlist}
\end{cor}

\begin{pf}
The first step is to expand $[u^n]
\frac{(1+u/m)^{km}}{(1-u/m)^{km}}$ as a series in $k$. One
calculates that
\begin{eqnarray*}
&&\bigl[u^n\bigr]\frac{(1+u/m)^{km}}{(1-u/m)^{km}}\\[-1pt]
&&\qquad=\frac1{m^n}
\sum_{a\geq0}\pmatrix{km
\cr
a}\pmatrix{km+n-a-1
\cr
n-a}
\\[-1pt]
&&\qquad=\frac1{m^n}\sum_{a\geq0} \biggl[
\frac{(km)\cdots(km-a+1)}{a!} \biggr] \\[-1pt]
&&\hspace*{42pt}\qquad{}\times\biggl[\frac{(km+n-a-1)\cdots
(km)}{(n-a)!} \biggr]
\\[-1pt]
&&\qquad=\frac1{n!} \biggl[2^nk^n+\frac{2^nn(n-1)(n-2)}{12m^2}k^{n-2}+
\cdots\biggr],
\end{eqnarray*}
where the $\cdots$ in the last equation denote terms of lower order in
$k$. Thus Theorem~\ref{desgen} gives
\begin{eqnarray*}
&&\sum_wP_m(w)t^{d(w)+1}\\[-1pt]
&&\qquad=
\biggl[\frac{(1-t)^{n+1}}{n!}\sum_{k\geq1}t^kk^n
\biggr]+ \frac{n-2}{12m^2}(1-t)^2 \biggl[\frac{(1-t)^{n-1}}{(n-2)!}\sum
_{k\geq1}t^kk^{n-2} \biggr]
\\[-1pt]
&&\qquad\quad{}+(1-t)^3C(t),
\end{eqnarray*}
where $C(t)$ is a polynomial in $t$. By equation \eqref{eul}, it
follows that
\[
\sum_wP_m(w)t^{d(w)+1}=
\frac{A_n(t)}{n!}+(1-t)^2\frac{n-2}{12m^2}\frac{A_{n-2}(t)}{(n-2)!}+(1-t)^3
C(t).
\]
Since the number of descents of a random permutation has mean
$(n-1)/2$ and variance $(n+1)/12$ for $n \geq2$, it follows that
$\frac{A_n'(1)}{n!} = \frac{(n+1)}{2}$ and also that
$\frac{A_n''(1)}{n!}= (3n^2+n-2)/12$. Thus
\[
\sum_wP_m(w)d(w)=\frac{n-1}2
\]
and
\[
\sum_w P_m(w) d(w)
\bigl[d(w)+1\bigr]=\frac{3n^2+n-2}{12}+\frac{n-2}{6m^2},
\]
and the result follows.\vadjust{\goodbreak}
\end{pf}

%
\begin{rems*}

\begin{itemize}
\item Part 1 of Corollary~\ref{descor} can be proved without
generating functions simply by noting that by the way the
shelf shuffler works, $w$ and its reversal are equally likely to be
produced.
\item Theorem~\ref{desgen} has an analog for ordinary riffle shuffles
which is useful in the study of carries in addition. See \citet{DFa}
for details.
\end{itemize}
\end{rems*}

\section{Iterated shuffling}\label{sec4}

This section shows how to analyze repeated shuffles. Section~\ref{newsec41}
shows how to combine shuffles. Section~\ref{newsec42} gives a clean bound for
the separation distance.

\subsection{Combining shuffles}\label{newsec41}

To describe what happens to various combinations of shuffles, we need
the notion of a signed $m$-shuffle. This has the following geometric
description: divide the unit interval into sub-intervals of length
$\frac1{m}$; each sub-interval contains the graph of a straight line
of slope $\pm m$. The left-to-right pattern of signs $\pm s$ is
indicated by a vector $x$ of length~$m$. Thus if $m=4$ and $x=++++$,
an $x$-shuffle is generated as shown on the left side of Figure~\ref{fig2}.
If $m=4$ and $x=+ - - +$, the graph becomes that of the right-hand side of
Figure~\ref{fig2}. Call this function $f_x$.

%
\begin{figure}

\includegraphics{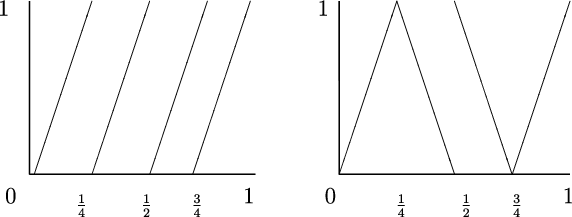}

\caption{{Left}: $m=4, x=+ + + +$; {Right}: $m=4, x=+ - - +$.}
\label{fig2}
\end{figure}

%
%
%
%

The shuffle proceeds as in the figure with $n$ points dropped at
random into the unit interval, labeled left to right,
$y_1,y_2,\ldots,y_n$ and then permuted by $f_x$. In each case there is
a simple forward description: the deck is cut into $m$ piles by a
multinomial distribution and piles corresponding to negative
coordinates are reversed. Finally, all packets are shuffled together
by the GSR procedure in which one drops each card sequentially with probability
proportional to packet size. Call the associated measure on permutations~$P_x$.

%
\begin{rem*}
Thus, ordinary riffle shuffles are $++$ shuffles. The shelf shuffle
with 10 shelves is an inverse $+-+-\cdots+-$ (length 20) shuffle in
this notation.
\end{rem*}

The following theorem reduces repeated shuffles to a single
shuffle. To state it, one piece of notation is needed. Let
$x=(x_1,x_2,\ldots,x_a)$ and $y=(y_1,y_2,\ldots,y_b)$ be two sequences
of $\pm$ signs. Define a sequence of length $ab$ as $x\ast
y=y^{x_1},y^{x_2},\ldots,y^{x_a}$ with
$(y_1,\ldots,y_b)^1=(y_1,\ldots,y_b)$ and
$(y_1,\ldots,y_b)^{-1}=(-y_b,-y_{b-1},\ldots,-y_1)$. This is an
associative product on strings; it is not commutative. Let $P_x$ be
the measure induced on $S_n$ (forward shuffles).

%
\begin{defn*}
\begin{eqnarray*}
(+++)\ast(++)&=&++++++,
\\
(+-)\ast(+-)&=&+-+-,
\\
(+-)\ast(++-+)&=&++-+-+--\!.
\end{eqnarray*}
\end{defn*}

%
\begin{thmm}\label{thm42}
If $x$ and $y$ are $\pm1$ sequences of length $a$ and $b$,
respectively, then
\[
P_x\ast P_y=P_{x\ast y}.
\]
\end{thmm}

\begin{pf}
In outline, this follows most easily from the geometric description underlying
Figures~\ref{fig1} and~\ref{fig2}. If a uniformly chosen point in
$[0,1]$ is
expressed base $a$, the ``digits'' are uniform and independently
distributed in $\{0,1,\ldots,a-1\}$. Because of this, iterating the
maps on the same uniform points gives the convolution. The iterated
maps have the claimed pattern of slopes by a simple geometric
argument.

In more detail, consider a $\pm$ string $x$ of length $m$.
The function $f_x$ sends $[0,1]$ to itself by mapping
$\eta\in[\frac{i-1}{m},\frac{i}{m}]$ to
$x_i m \eta$ (mod 1). If $x_i$ is positive, all points
in $[\frac{i-1}{m},\frac{i}{m}]$
are sent to
$[0,1]$ in an order preserving way. If $x_i$ is negative, the order
is reversed. In either case, the map $f_x$ is $m$ to $1$ and measure preserving
on $[0,1]$ (i.e., $f_x^{-1}$ preserves Lebesgue measure).
Now consider $f_y \circ f_x$ (apply $f_x$ first, then $f_y$),
where $x$ is of length $m$ and $y$ is of length $n$. The composition
sends all elements $\eta\in[\frac{i-1}{mn},\frac{i}{mn}]$
to $\operatorname{sgn}(i)$ $m n \eta$ (mod 1); where $\operatorname{sgn}(i$) is $\pm1$. A simple
argument shows $\operatorname{sgn}(i)$ is given
by the $x*y$ rule.

There is a subtle point: the convolution $P_x*P_y$
(first apply $P_x$ then independently $P_y$)
involves independent shuffles while
$f_y \circ f_x$ and $f_{x*y}$ are applied to a single random uniform
set of points. The induced measures are the same
because the digits of $\eta\in[0,1]$
base $m$ are independent and uniformly distributed in
$ \{ 0, 1, \ldots, m \}$
and $f_x$ preserves measure.
It follows that the image of independent uniform points in $ [0,1]$
under $f_x$ are independent and uniform even conditional
on the induced permutation which is determined by how many points fall
in each $[\frac{i-1}{m},\frac{i}{m}]$.
\end{pf}

%
\begin{cor}
The convolution of $k +-$ shuffles is a $+-+-\cdots+-$ ($2^k$ terms)
shuffle. Further, the convolution of a shelf shuffler with $m_1$ and
then $m_2$ shelves is the same as a shelf shuffler with $2m_1m_2$
shelves.
\label{cor43}
\end{cor}

\subsection{Bounds for separation distance}\label{newsec42}

The following theorem gives a bound for separation (and so for total
variation) for a general $P_x$ shuffle on $S_n$.

%
\begin{thmm}\label{thm44}
For any $\pm1$ sequence $x$ of length $a$, with $P_x$ the associated
measure on $S_n$, and $\separ(P_x)$ from \eqref{pound},
%
\begin{equation}
\separ(P_x)\leq1-\prod_{i=1}^{n-1}
\biggl(1-\frac{i}{a} \biggr). \label{41}
\end{equation}
\end{thmm}

\begin{pf}
It is easiest to argue using shuffles as in Description~\ref{desc1}.
There, the
backs of cards are labeled, independently and uniformly, with
symbols $1,2,\ldots,a$. For the inverse shuffle, all cards labeled 1
are removed, keeping them in their same relative order, and placed
on top followed by the cards labeled 2 (placed under the~1s) and so
on, with the following proviso: if the $i$th coordinate of $x$ is
$-1$, the cards labeled $i$ have their order reversed; so if they
are 1, 5, 17 from top down, they are placed in order 17, 5, 1. All
of this results in a single permutation drawn from $P_x$. Repeated
shuffles are modeled by labeling each card with a vector of symbols.
The $k$th shuffle is determined by the $k$th coordinate of this
vector. The first time $t$ that the first $k$ coordinates of those
$n$ vectors are all distinct forms a strong stationary time. See
\citet{AD} or \citet{fulman98} for further details. The usual bound
for separation yields
\[
\separ(P_x)\leq P\{\mbox{all $n$ labels are distinct}\}.
\]
The bound \eqref{41} now follows from the classical birthday problem.
\end{pf}

%
\begin{rems*}

\begin{itemize}
\item For $a$ large with respect to $n$, the right-hand side is
well-approximated by $1-e^{-({n(n-1)})/{(2a)}}$. This is small when
$n^2\ll a$.
\item The theorem gives a clean upper bound on the distance to
uniformity. For example, when $n=52$, after $8$ ordinary riffle
shuffles (so $x=++\cdots++$, length 256), the bound \eqref{41} is
$\separ(P_x)\leq0.997$, in agreement with Table~1 of \citet{assaf}. For
the actual shelf shuffle with $x=+-+-\cdots+-$ (length 20), the bound
gives $\separ(P_x)=1$ but $\separ(P_x\ast P_x)\leq0.969$ and
$\separ(P_x\ast P_x\ast P_x)\leq0.153$.
\item The bound in Theorem~\ref{thm44} is simple and general. However,
it is not sharp for the original shelf shuffler. The results of
Section~\ref{newsec33} show that $m=cn^{3/2}$ shelves suffice to make
$\separ(P_m)$ small when $c$ is large. Theorem~\ref{thm44} shows that
$m=cn^2$ steps suffice.
\item
The upper bound in \eqref{41} is achieved. If the length $a$ sequence
consists only of $+$ signs, we have the ordinary ``riffle shuffle.''
Then the formula in Bayer and Diaconis (\citeyear{bayer}) for the chance of a
permutation after an $a$ shuffle implies that the separation distance
is attained for the permutation $n, n-1,\ldots, 1$, and is equal to
the right-hand side of
\eqref{41}.

\end{itemize}
\end{rems*}

\section{Practical tests and conclusions}\label{sec5}

The engineers and executives who consulted us found it hard to
understand the total variation distance. They asked for more
down-to-earth notions of discrepancy. This section reports some ad hoc
tests which convinced them that the machine had to be used
differently. Section~\ref{sec51} describes the number of cards guessed
correctly. Section~\ref{sec52} briefly describes three other tests.
Section~\ref{sec53} describes conclusions and recommendations.

\subsection{Card guessing with feedback}\label{sec51}

Suppose, after a shuffle, cards are dealt face-up, one at a time, onto
the table. Before each card is shown, a guess is made at the value of
the card. Let $X_i, 1\leq i\leq n$, be one or zero as the $i$th guess
is correct and $T_n=X_1+\cdots+X_n$ the total number of correct
guesses. If the cards were perfectly mixed, the chance that $X_1=1$ is
$1/n$, the chance that $X_2=1$ is $1/(n-1), \ldots,$ that $X_i=1$
is $1/(n-i+1)$. Further, the $X_i$ are independent. Thus elementary
arguments give the following.

%
\begin{prop}\label{prop51}
Under the uniform distribution, the number of cards guessed correctly
$T_n$ satisfies:
\begin{itemize}
\item$E(T_n)=\frac1{n}+\frac1{n-1}+\cdots+1\sim\log
n+\gamma+O (\frac1{n} )$ with $\gamma\doteq0.577$ Euler's constant.
\item$\var(T_n)=\frac1{n} (1-\frac1{n} )+\frac1{n-1} (1-\frac1{n-1}
)+\cdots+
\frac12 (1-\frac12 )\sim\log n+\gamma-\frac{\pi^2}6+O (\frac1{n} )$.
\item Normalized by its mean and variance, $T_n$ has an approximate
normal distribution.
\end{itemize}
\end{prop}

When $n=52$, $T_n$ has mean approximately 4.5, standard deviation
approximately $\sqrt{2.9}$ and the number of correct guesses is
between 2.7 and 6.3, 70\% of the time.

Based on the theory developed in Section~\ref{sec3} we constructed a guessing
strategy---conjectured to be optimal---for use after a shelf
shuffle.

\textit{Strategy}.
\begin{itemize}
\item To begin, guess card 1.
\item If guess is correct, remove card 1 from the list of available
cards. Then guess card 2, card 3, $\ldots.$
\item If guess is incorrect and card $i$ is shown, remove card $i$
from the list of available cards and guess card $i+1$, card $i+2,
\ldots.$
\item Continue until a descent is observed (order reversal with the
value of the current card smaller than the value of the previously
seen card). Then change the guessing strategy to guess the
next-smallest available card.
\item Continue until an ascent is observed, then guess the
next-largest available card, and so on.
\end{itemize}

%
\begin{table}
\caption{Mean and variance for $n=52$ after a shelf shuffle with $m$
shelves under the conjectured optimal strategy}\label{table3}\vspace*{-3pt}
\begin{tabular*}{\textwidth}{@{\extracolsep{\fill}}ld{2.1}d{2.1}d{2.1}ccc@{}}
\hline
$\bolds{m}$&\multicolumn{1}{c}{\textbf{1}}&\multicolumn{1}{c}{\textbf{2}}&\multicolumn{1}{c}{\textbf{4}}&
\multicolumn{1}{c}{\textbf{10}}&\multicolumn{1}{c}{\textbf{20}}&\multicolumn{1}{c@{}}{\textbf{64}}\\
\hline
Mean &39&27&17.6&9.3&6.2&4.7\\
Variance&3.2&5.6&6.0&4.7&3.8&3.1\\
\hline
\end{tabular*}  \vspace*{-3pt}
\end{table}

%
\begin{figure}[b]\vspace*{-3pt}

\includegraphics{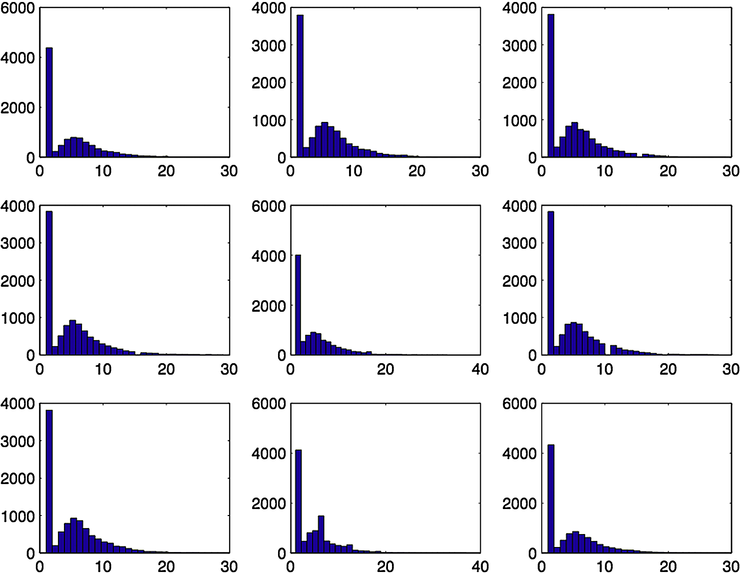}\vspace*{-3pt}

\caption{9 spacings from a 10-shelf shuffle; $j$ varies from top left
to bottom right, $1\leq j\leq9$.}
\label{fig3}
\end{figure}

A Monte Carlo experiment was run to determine the distribution of
$T_n$ for \mbox{$n=52$} with various values of $m$ (10,000 runs for each
value). Table~\ref{table3} shows the mean and variance for various numbers\vadjust{\goodbreak}
of shelves. Thus for the actual shuffler, $m=10$ gives about 9.3
correct guesses versus 4.5 for a well-shuffled deck. A~closely related
study of optimal strategy for the GSR measure (without feedback) is
carried out by \citet{ciucu}.\vspace*{-3pt}

\subsection{Three other tests}\label{sec52}

For the shelf shuffler with $m$ shelves, an easy argument shows that
the chance that the original top card is still on top is at least
$1/2m$ instead of $1/n$. When $n=52$, this is $1/20$ versus $1/52$. The
chance that card 2 is on top is approximately
$\frac1{2m} (1-\frac1{2m} )$ while the chance that card 2 is
second from the top is roughly $\frac1{(2m)^2}$. The same
probabilities hold for the bottom cards. While not as striking as the
guessing test of Section~\ref{sec51}, this still suggests that the
machine is
``off.''\vadjust{\goodbreak}

Our second test supposed that the deck was originally arranged with
all the red cards on top and all the black cards at the bottom. The
test statistic is the number of changes of color going through the
shuffled deck. Under uniformity, simulations show this has mean 26 and
standard deviation~3.6. With a 10-shelf machine, simulations showed
$17\pm1.83$, a noticeable deviation.
The third test is based on the spacings between cards originally near
the top of the deck. Let $w_j$ denote the position of the card
originally at position $j$ from the top. Let $D_j=|w_j-w_{j+1}|$.
Figure~\ref{fig3} shows a histogram of $D_j$ for $1\leq j\leq9$, from a
simulation with $n=52$ based on a 10-shelf shuffler. Figure~\ref{fig4} shows
histograms for the same statistics for a well-shuffled deck; there are
striking discrepancies.\vspace*{-3pt}

%
\begin{figure}

\includegraphics{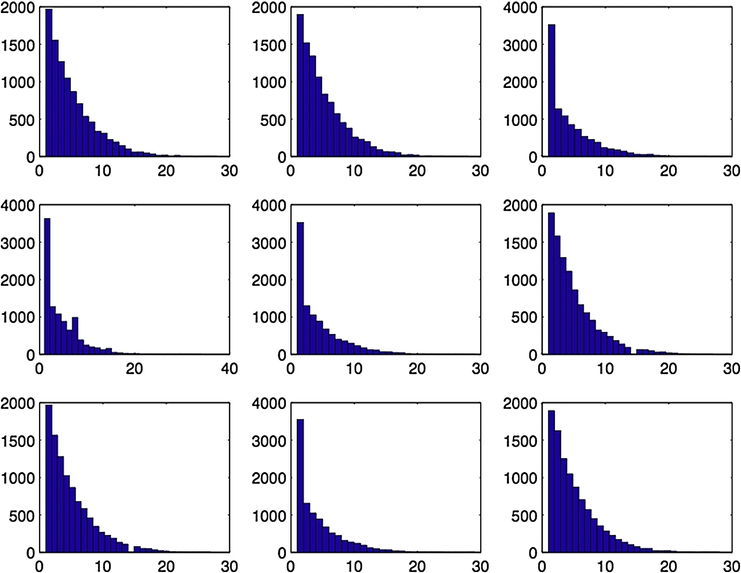}\vspace*{-3pt}

\caption{9 spacings from a uniform shuffle; $j$ varies from top left
to bottom right, $1\leq j\leq9$.}
\label{fig4}\vspace*{-3pt}
\end{figure}

\subsection{Conclusions and recommendations}\label{sec53}

The study above shows that a single iteration of a 10-shelf shuffler
is not sufficiently random. The president of the company responded,
``We are not pleased with your conclusions, but we believe them and
that's what we hired you for.''

We suggested a simple alternative: use the machine twice. This results
in a shuffle equivalent to a 200-shelf machine. Our mathematical
analysis and further tests, not reported here, show that this\vadjust{\goodbreak}
\textit{is} adequately random. Indeed, Table~\ref{table1} shows, for total
variation, this is equivalent to 8-to-9 ordinary riffle
shuffles.\vspace*{-3pt}

\section*{Acknowledgments}
We thank the Editor and two anonymous referees for their constructive reviews.\vspace*{-3pt}

%


\printaddresses

\end{document}